\documentclass[12pt]{amsart}
\usepackage{etex}
\usepackage{amsfonts, amssymb, latexsym, epsfig} 

\usepackage{amscd,amssymb}
\usepackage{verbatim}
\usepackage[mathscr]{euscript}
\usepackage{amsmath}
\usepackage{amsthm}
\input xy
\xyoption{all}

\newsymbol\pp 1275

\usepackage{tikz}
\usetikzlibrary{matrix,arrows,backgrounds,shapes.misc,shapes.geometric,patterns,calc,positioning}

\usepackage[colorlinks=true, pdfstartview=FitV, linkcolor=blue, citecolor=blue, urlcolor=blue]{hyperref}

\usepackage{fullpage}
\usepackage{graphicx}
\usepackage{enumerate}
\def\VR{\kern-\arraycolsep\strut\vrule &\kern-\arraycolsep}
\def\vr{\kern-\arraycolsep & \kern-\arraycolsep}

\newcommand{\dv}{\ensuremath{\underline{\dim}\, }}

\newcommand{\be}{\begin{enumerate}}

\newtheorem{theorem}{Theorem}

\newtheorem{lemma}[theorem]{Lemma}
\newtheorem{prop}[theorem]{Proposition}
\newtheorem{corollary}[theorem]{Corollary}

\theoremstyle{definition}

\newtheorem{rmk}{Remark}
\newenvironment{remark}[1][]{\begin{rmk}[#1]\pushQED{\qed}}{\popQED \end{rmk}}

\newtheorem{obs}{Observation}

\newtheorem{ex}{Example}

\newcommand{\Hom}{\operatorname{Hom}}
\newcommand{\End}{\operatorname{End}}
\newcommand{\Ext}{\operatorname{Ext}}

\newcommand{\rep}{\operatorname{rep}}

\newcommand{\ZZ}{\mathbb Z}

\newcommand{\QQ}{\mathbb Q}

\newcommand{\rel}{\operatorname{relint}}

 \newcommand{\ee}{\operatorname{\mathbf{e}}}

\newcommand{\C}{\mathcal{C}}
\newcommand{\D}{\mathcal{D}}
\newcommand{\I}{\mathcal{I}}

\newcommand{\wt}{\operatorname{wt}}
\newcommand{\module}{\operatorname{mod}}

\newcommand{\git}{\operatorname{GIT}}

\newcommand{\rat}{\mathbb{Q}}

\newcommand{\nat}{\mathbb{N}}
\newcommand{\inte}{\mathbb{Z}}

\newcommand{\into}{\hookrightarrow} 

\DeclareMathOperator{\relint}{relint}

\DeclareMathOperator{\Du}{D}

\newcommand{\<}{\langle} 
\renewcommand{\>}{\rangle} 

\theoremstyle{plain}

\newtheorem*{thm*}{Theorem}

\theoremstyle{definition}

\theoremstyle{remark}
\newtheorem*{re}{Remark}

\begin{document}
\title{GIT-equivalence and Semi-Stable Subcategories of Quiver Representations}

\author{Calin Chindris}
\address{University of Missouri-Columbia, Mathematics Department, Columbia, MO, USA}
\email[Calin Chindris]{chindrisc@missouri.edu}

\author{Valerie Granger}
\address{University of Missouri-Columbia, Mathematics Department, Columbia, MO, USA}
\email[Valerie Granger]{vrhmq7@mail.missouri.edu}

\date{\today}
\bibliographystyle{plain}
\subjclass[2000]{16G20, 13A50}
\keywords{semi-stable quiver representations, GIT-cones, Schur roots, tame quivers}

\begin{abstract} In this paper, we answer the question of when the subcategory of semi-stable representations is the same for two rational vectors for an acyclic quiver. This question has been previously answered by Ingalls, Paquette, and Thomas in the tame case in \cite{IngPacTho}. Here we take a more invariant theoretic approach, to answer this question in general. We recover the known result in the tame case. 
\end{abstract}

\maketitle
\setcounter{tocdepth}{1}
\tableofcontents

\section{Introduction}
Throughout, $K$ denotes an algebraically closed field of characteristic zero. By a quiver, we mean a finite, connected, acyclic quiver. All representations and modules are assumed to be finite-dimensional. By a module, we always mean a left module.

Our goal in this paper is to determine when two rational vectors of a quiver $Q$ share the same subcategory of semi-stable representations. This question has been answered by Ingalls, Paquette, and Thomas in the tame case in \cite{IngPacTho}. Here we take a representation-type free approach, to answer this question in general. We recover the known result in the tame case. We point out that the  study of the possible interactions between semi-stable subcategories is important for applications to cluster algebras and Artin groups  (see \cite{Brady, BradyWatt}, \cite{IngTho}).

Let $Q$ be a quiver with vertex set $Q_0$
and consider the isomorphism of vector spaces:
\begin{align*}
\wt :\QQ^{Q_0}&\to \QQ^{Q_0}\\
\alpha& \to (\langle \alpha, \ee_x \rangle)_{x \in Q_0}, 
\end{align*}
where the $\ee_x$, $x \in Q_0$, are the standard unit vectors of $\QQ^{Q_0}$ and $\langle \cdot, \cdot \rangle$ is the Euler inner product of $Q$. 
Two rational vectors $\alpha_1$ and $\alpha_2 \in \QQ^{Q_0}$ are said to be {\bf$\git$-equivalent} if they share the same subcategory of semi-stable representations, i.e., if: 
$$\rep(Q)_{\wt(\alpha_1)}^{ss} = \rep(Q)_{\wt(\alpha_2)}^{ss}.$$ 
 
The {\bf $\git$-cone} associated to a dimension vector $\beta \in \ZZ^{Q_0}_{\geq 0}$ and a rational vector $\alpha \in \QQ^{Q_0}$ is 
$$\C(\beta)_{\alpha} = \{ \alpha' \in \D(\beta) | \rep(Q, \beta)_{\wt(\alpha)}^{ss} \subseteq \rep(Q, \beta)_{\wt(\alpha')}^{ss}\},$$
where $\D(\beta)$ is the domain of semi-invariants associated to $Q$ and $\beta$ (we refer to Section \ref{dom} for the details behind our notation).

For each Schur root $\beta$ of $Q$, define $\mathcal{C}(\beta)$ to be the collection of maximal (with respect to dimension) GIT-cones of $\D(\beta)$. Let
$$\I = \bigcup \C(\beta),$$ 
where the union is over all Schur roots $\beta$ of $Q$. For any rational vector $\alpha \in \QQ^{Q_0}$, define: 
$$\I_{\alpha} = \{ \C \in \I \mid \alpha \in \C \}.$$

Now we are ready to state our first result:

\begin{theorem} \label{thm1} Let $Q$ be a quiver with vertex set $Q_0$ and let $\alpha_1$ and $\alpha_2$ be two rational vectors in $\QQ^{Q_0}$. Then $\alpha_1$ and $\alpha_2$ are $\git$-equivalent if and only if $\I_{\alpha_1}=\I_{\alpha_2}$. 
\end{theorem}

Let us assume now that $Q$ is tame and let $N$ be the number of non-homogeneous regular tubes in the Auslander-Reiten quiver of $Q$. Let $r_i$ be the rank of the $i^{th}$ tube, let $I = (a_i)_{i=1}^{N}$ with $1\leq a_i \leq r_i$, and let $R$ be the set of all such multi-indices $I$. Let $\{\beta_{i,j}\}_{i = 1, \ldots, N, j = 1, \ldots, r_i}$ be the regular simple roots in the non-homogeneous tubes, and set: $$\alpha_I:=\delta + \sum_{i=1}^n \sum_{j\neq a_i} \beta_{i,j},$$
i.e. $\alpha_I$ is the sum of $\delta$, and all other regular simples {\it except} $\beta_{i,a_i}$ for $i = 1, \ldots, N$. Now, we are ready to state our second result, which recovers Theorem 7.4 in \cite{IngPacTho}.

\begin{theorem} \label{thm2} Let $Q$ be a tame quiver. Then: 
$$\I = \{\C(\delta)_{\alpha_I}\}_{I \in R} \cup \{\D(\beta)\}_{\beta},$$
where $\delta$ is the unique isotropic Schur root and the $\beta$'s run through all real Schur roots of $Q$.
\end{theorem}

\noindent
\textbf{Acknowledgment:} The second author would like to thank Charles Paquette for clarifying discussions on the results in \cite{IngPacTho}. She is also thankful to Dan Kline for his assistance with the proof of Lemma \ref{dan}. The authors were supported by the NSA under grant H98230-15-1-0022. 

\section{Background} \label{background-sec}

Let $Q=(Q_0,Q_1,t,h)$ be a finite quiver with vertex set $Q_0$ and arrow set $Q_1$. The two functions $t,h:Q_1 \to Q_0$ assign to each arrow $a \in Q_1$ its tail \emph{ta} and head \emph{ha}, respectively.

A {\bf representation} $V$ of $Q$ over $K$ is a collection $(V(x),V(a))_{x\in Q_0, a\in Q_1}$ of finite-dimensional $K$-vector spaces $V(x)$, $x \in Q_0$, and $K$-linear maps $V(a): V(ta) \to V(ha)$, $a \in Q_1$. The dimension vector of a representation $V$ of $Q$ is the function $\dv V \colon Q_0 \to \ZZ$ defined by $(\dv V)(x)=\dim_{K} V(x)$ for $x\in Q_0$. The one-dimensional representation of $Q$ supported at vertex $x \in Q_0$ is denoted by $S_x$ and its dimension vector is denoted by $\ee_x$. By a dimension vector of $Q$, we simply mean a vector $\beta \in \ZZ_{\geq 0}^{Q_0}$.

Let $V$ and $W$ be two representations of $Q$. A {\bf morphism} $\varphi:V \rightarrow W$ is defined to be a collection $(\varphi(x))_{x \in Q_0}$ of $K$-linear maps with $\varphi(x) \in \Hom_K(V(x), W(x))$ for each $x \in Q_0$, such that $\varphi(ha)V(a)=W(a)\varphi(ta)$ for each $a \in Q_1$. We denote by $\Hom_Q(V,W)$ the $K$-vector space of all morphisms from $V$ to $W$. We say that $V$ is a subrepresentation of $W$ if $V(x)$ is a subspace of $W(x)$ for each $x \in Q_0$ and $(i_x: V(x) \hookrightarrow W(x))_{x \in Q_0}$ is a morphism of representations, i.e. $V(a)$ is the restriction of $W(a)$ to $V(ta)$ for each $a \in Q_1$. The category of all representations of $Q$ is denoted by $\rep(Q)$. It turns out that $\rep(Q)$ is an abelian hereditary category. A representation $V \in \rep(Q)$ is called a {\bf Schur representation} if $\End_Q(V) \cong K$. A Schur representation $V \in \rep(Q)$ is called {\bf exceptional} if $\Ext_Q^1(V,V) = 0$. 

The path algebra, $KQ$, of $Q$ has a $K$-basis consisting of all paths (including the trivial ones), and multiplication in $KQ$ is given by concatenation of paths. It is easy to see that any $KQ$-module defines a representation of $Q$, and vice-versa. Furthermore, the category $\module(KQ)$ of $KQ$-modules is equivalent to the category $\rep(Q)$. In what follows, we identify $\module(KQ)$ and $\rep(Q)$, and use the same notation for a module and the corresponding representation.

The Euler form of $Q$ is the bilinear form $\langle -,- \rangle: \ZZ^{Q_0} \times \ZZ^{Q_0} \to \ZZ$ defined by
$$
\langle \alpha, \beta \rangle=\sum_{x \in Q_0} \alpha(x)\beta(x)-\sum_{a \in Q_1} \alpha(ta) \beta(ha), \forall \alpha, \beta \in \ZZ^{Q_0}.
$$
This extends to a bilinear form on $\rat^{Q_0} \times \rat^{Q_0}$. The corresponding Tits quadratic form is $q:\ZZ^{Q_0}\to \ZZ$, $q(\alpha)=\langle \alpha, \alpha \rangle, \forall \alpha \in \ZZ^{Q_0}$. 

For any two representations $V, W \in \rep(Q)$, we have (see \cite{R}): 
$$
\langle \dv V, \dv W \rangle=\dim_K \Hom_Q(V,W)-\dim_K \Ext^1_{Q}(V,W).
$$
\noindent
Since $Q$ is assumed to be acyclic, the linear map: 
\begin{align*}
\rat^{Q_0} & \to (\rat^{Q_0})^* \\
\alpha & \mapsto \<\alpha, \cdot\>
\end{align*}
is an isomorphism. In fact, composing this map on the right with the canonical isomorphism $(\rat^{Q_0})^* \to \rat^{Q_0}$, we get the isomorphism $\wt:\rat^{Q_0} \to \rat^{Q_0}$. Given $\theta, \beta \in \rat^{Q_0}$, we denote by $\theta(\beta) := \sum_{x\in Q_0} \theta(x)\beta(x)$. Then for any $\alpha, \beta \in \rat^{Q_0}$, we have that 
$$\wt(\alpha)(\beta)=\< \alpha, \beta \>.$$

We call a dimension vector $\beta$ a (Schur, exceptional) {\bf root} of $Q$ if there exists a $\beta$-dimensional (Schur, exceptional) indecomposable representation. If $\beta$ is a root of $Q$, then $q(\beta) \leq 1$. More specifically, a root $\beta$ is said to be {\bf real} if $q(\beta) =1$. We call $\beta$ an (imaginary) {\bf isotropic root} if $q(\beta)=0$. If $q(\beta) <0$ we say that $\beta$ is {\bf strictly imaginary}.

One of the fundamental results in the representation theory of finite-dimensional algebras is the {\bf Auslander-Reiten formula}, which we state for quiver representations (see for example \cite{BB1}): For two indecomposable representations $V$ and $W$ of $Q$, we have: 
$$
\Hom_Q(V, \tau W) \cong \Hom_Q(\tau^-V, W) \cong \Du\Ext_Q^1(W,V),
$$
where $\tau$ is the Auslander-Reiten translation of $Q$.

\subsection{Domains of semi-invariants of quivers} \label{dom}
Let $\theta \in \QQ^{Q_0}$. A representation $V \in \rep(Q)$ is said to {\bf $\theta$-semi-stable} if:
$$
\theta(\dv V)=0 \text{~and~} \theta(\dv V')\leq 0, \forall {\ } V' \leq V.
$$
We say that $V$ is $\theta$-stable if $V$ is not the zero representation, $\theta(\dv V)=0$, and $\theta(\dv V')<0$ for all proper subrepresentations $0 \neq V' < V$. We define $\rep(Q)^{ss}_{\theta}$ to be the full subcategory of $\rep(Q)$ whose objects are the $\theta$-semi-stable representations of $Q$. Note that the simple objects of $\rep(Q)^{ss}_{\theta}$ are precisely the $\theta$-stable representations. Furthermore, $\rep(Q)^{ss}_{\theta}$ is an abelian subcategory of $\rep(Q)$, closed under extensions. It is also Artinian and Noetherian and hence any $V \in \rep(Q)^{ss}_{\theta}$ has a filtration whose factors are $\theta$-stable. 

Let $\beta$ be a dimension vector of $Q$. We say that a dimension vector $\beta$ is  $\theta$-(semi)-stable if there exists a $\beta$-dimensional $\theta$-(semi)-stable representation. We define the {\bf domain of semi-invariants} associated to $Q$ and $\beta$ to be: 
$$
\D(\beta) = \{\alpha \in \rat^{Q_0} | \beta \mbox{ is } \wt(\alpha)\mbox{-semi-stable}\}.
$$

In what follows, we write $\beta' \hookrightarrow \beta$ to mean that any $\beta$-dimensional representation has a subrepresentation of dimension vector $\beta'$. Then, according to King's criterion for semi-stability of quiver representations (see  \cite{KSC}), we have:
$$
\D(\beta)=\{\alpha \in \QQ^{Q_0} \mid \<\alpha,\beta\>=0, \<\alpha, \beta' \> \leq 0, \forall{\ } \beta' \hookrightarrow \beta\}.
$$
This makes it clear that $\D(\beta)$ is a rational convex polyhedral cone. Furthermore, it follows from Schofield's results in \cite{S2, S1} that:
$$
\D(\beta) \cap \ZZ_{\geq 0}^{Q_0}=\{\alpha \in \ZZ_{\geq 0}^{Q_0} \mid \alpha \mbox{ is } -\< \cdot, \beta \>-\mbox{semi-stable}\}.
$$ 

Next, we give a description of $\D(\beta)$ in terms of generators. We say that a dimension vector $\alpha$ is {\bf $\beta$-simple} if $\alpha$ is $-\<\cdot, \beta\>$-stable. Recall that the projective cover of the simple representation $S_x$ is denoted $P_x$, and its dimension vector is denoted by $\gamma_x$, for $x \in Q_0$. For any $\rho \in \ZZ_{\geq 0}^{Q_0}$, we can construct a projective representation $P_{\rho} = \bigoplus_{x \in Q_0} P_{x}^{\rho(x)}$. 

\begin{lemma} \label{simples} Let $\beta$ be a dimension vector. Then $\D(\beta)$ is generated by the $\beta$-simple roots, together with $-\gamma_x$ for $x \in Q_0$ such that $\beta(x) = 0$.   
\end{lemma}

\begin{proof} Let $\alpha \in \D(\beta)$. If $\alpha$ is not integral, multiply by a large enough integer to obtain an integral vector. 

Define the following projective representation: $P_\rho$, where $\rho(x) = 1$ if $\beta(x) = 0$, and $\rho(x) = 0$ otherwise. From \cite[Lemma 6.5.7]{IOTW}, we know that for each $y \in Q_0$ with $\alpha(y)<0$, there exists an $x \in Q_0$ with $\beta(x) = 0$ and $\dim_K P_x(y) >0$. Thus, there exists a positive integer $m_y$ such that $\alpha(y) + m_y \dim_K P_x(y) \geq 0$. Taking a sufficiently large $m \gg 0$, then, we have: 

$$\alpha + \dv P_{\rho}^m \in \inte_{\geq 0}^{Q_0}.$$

Now, $\dv P_{\rho} \in \D(\beta)$, because $\<\dv P_x, \beta \> = \beta(x) = 0$ and, if $\beta' \into \beta$, then $\<\dv P_x, \beta'\> = \beta'(x) \leq \beta(x) = 0$ for each $x$ with $\rho(x) = 1$. That is, $\beta$ is $\<\dv P_x,\cdot\>$-semi-stable for each $x$ with $\beta(x) = 0$. 

So, $\alpha + \dv P_{\rho}^m \in \D(\beta) \cap \inte_{\geq 0}^{Q_0}$, which means there exists a $V$ with $\dv V = \alpha + \dv P_{\rho}^m$, such that $V$ is $-\<\cdot, \beta\>$-semi-stable. Next, consider a Jordan-H\"older filtration of $V$ in the category $\rep(Q)_{-\<\cdot, \beta\>}^{ss}$: 
$$0=V_0 < V_1 < \cdots < V_{r-1} < V_r=V,$$
where the composition factors $V_i/V_{i-1}$ are $-\<\cdot, \beta\>$-stable; in particular, $\dv V_i/V_{i-1}$ is $\beta$-simple, for all $1 \leq i \leq r$. So,
$$\dv V = \alpha + \dv P_{\rho}^m = \sum_{i=1}^r \dv V_{i}/V_{i-1}$$
is a sum of $\beta$-simples. Finally, we see that $\alpha = \dv V - \dv P_{\rho}^m$ is a sum of $\beta$-simples, plus $-\gamma_x$'s for $x$ such that $\beta(x) = 0$. 
\end{proof}

A description of the faces of $\D(\beta)$ was found by Derksen and Weyman in \cite{DW2}. Specifically, let $\mathcal F$ be a (non-trivial) face of $\D(\beta)$ and choose a lattice point $\alpha$ in the relative interior of $\mathcal F$. It turns out that there exist unique $\wt(\alpha)$-stable dimension vectors $\beta_1, \ldots, \beta_l$ such that a generic representation in $\rep(Q,\beta)$ has a filtration (in $\rep(Q)^{ss}_{\wt(\alpha)}$) whose factors are $\wt(\alpha)$-stable representations of dimension $\beta_1, \ldots, \beta_l$ (in some order). We also write
$$
\beta=\beta_1 \pp \ldots \pp \beta_l
$$
and call this the $\wt(\alpha)$-stable decomposition of $\beta$. It is proved in \cite{DW2} that:
$$
\mathcal F=\D(\beta_1) \cap \ldots \cap \D(\beta_l).
$$

For any $\alpha \in \D(\beta)$, the GIT-class of $\alpha$ relative to $\beta$ is the set $\{ \alpha'  \in \rat^{Q_0} | \rep(Q,\beta)_{\wt(\alpha)}^{ss} = \rep(Q, \beta)_{\wt(\alpha')}^{ss}\}$. The following lemma will be needed in the proof of Proposition \ref{pure-fan}.

\begin{lemma} \label{L5} Let $\mathcal F$ be a face of $\D(\beta)$ and $\alpha \in \QQ^{Q_0}$ a rational vector lying in the relative interior of $\mathcal F$. Then the GIT-class of $\alpha$ relative to $\beta$ is included in $\mathcal F$.
\end{lemma}

\begin{proof} Let
$$
\beta=\beta_1 \pp \ldots \pp \beta_l
$$
be the $\wt(\alpha)$-stable decomposition of $\beta$. Then:
$$
\mathcal F=\D(\beta_1) \cap \ldots \cap \D(\beta_l).
$$
For each $1 \leq i \leq l$, choose $V_i \in \rep(Q,\beta_i)^s_{\wt(\alpha)}$ and set $V:=\bigoplus_{i=1}^l V_i \in \rep(Q,\beta)^{ss}_{\wt(\alpha)}$. 

Now, let $\alpha'$ be a rational vector such that $\rep(Q,\beta)^{ss}_{\wt(\alpha)}=\rep(Q,\beta)^{ss}_{\wt(\alpha')}$. Then $V \in \rep(Q,\beta)^{ss}_{\wt(\alpha')}$ and so each $V_i$ becomes $\wt(\alpha')$-semi-stable, i.e. $\alpha' \in \D(\beta_i)$ for all $1 \leq i \leq l$. This shows that $\alpha' \in \mathcal F$. 
\end{proof}

\section{GIT-equivalence for rational vectors} 

Let $Q$ be an arbitrary quiver. Given a rational vector $\alpha \in \QQ^{Q_0}$, recall that $\rep(Q)^{ss}_{\wt(\alpha)}$ is a full subcategory of $\rep(Q)$ whose class of objects consists of all $\wt(\alpha)$-semi-stable representations of $Q$. 

We say that two vectors  $\alpha_1, \alpha_2 \in \QQ^{Q_0}$ are {\bf $\git$-equivalent}, and write $\alpha_1 \sim_{\git} \alpha_2$, if 
$$
\rep(Q)^{ss}_{\wt(\alpha_1)}=\rep(Q)^{ss}_{\wt(\alpha_2)}.
$$ 

In what follows, we give a characterization of this equivalence relation in terms of certain rational convex polyhedral cones of codimension one in $\QQ^{Q_0}$. These cones arise most naturally in the context of variation of GIT-quotients. As we will see in Section \ref{tame}, in case $Q$ is tame, we can explicitly describe these cones and show they coincide with those found by Ingalls-Paquette-Thomas in \cite{IngPacTho}.

We know from \cite[Theorem 6.1]{S1} that $\beta$ is a Schur root if and only if $\D^0(\beta):=\{\alpha \in \QQ^{Q_0} \mid \langle \alpha, \beta \rangle =0, \langle \alpha, \beta' \rangle < 0, \forall \beta' \hookrightarrow \beta, \beta'\neq \mathbf{0}, \beta\}$ is non-empty. In this case, $\D^0(\beta)$ is precisely the relative interior of $\D(\beta)$ and $\dim_K \D(\beta)=|Q_0|-1$.

For a vector $\alpha \in \QQ^{Q_0}$, consider the set:
$$
\C(\beta)_{\alpha}:=\{\alpha'\in \D(\beta) \mid \rep(Q,\beta)^{ss}_{\wt(\alpha)} \subseteq \rep(Q,\beta)^{ss}_{\wt(\alpha')}\}.
$$  
We point out that this a rational convex polyhedral subcone of $\D(\beta)$, called the \emph{GIT-cone of $\alpha$ relative to $\beta$} (for more details, see \cite{CC5}). 

Finally, we define the {\bf $\git$-fan associated to $(Q,\beta)$} to be:
$$
\mathcal F(\beta):=\{\C(\beta)_{\alpha} \mid \alpha \in \D(\beta)\} \cup \{\mathbf{0}\}.
$$

\begin{re} $(1)$ The addition of the trivial cone $\mathbf{0}$ to the collection of GIT-cones above is needed only when $\beta$ is not sincere. Indeed, if $\beta$ is sincere the GIT-cone corresponding to the zero vector is the zero cone. But if $\beta$ is not sincere then any GIT-cone $\C(\beta)_{\alpha}$ has dimension at least the number of vertices of $Q$ where $\beta$ is zero.
\smallskip

\noindent
$(2)$ Note that $\mathcal{F}(\beta) \cap \I = \C(\beta)$.
\end{re}

\begin{theorem}\label{git-fan} Keeping the same notation as above, $\mathcal F(\beta)$ is a finite  fan cover of $\mathcal D(\beta)$.
\end{theorem}

\begin{re}
A proof of this result within the context of quiver invariant theory, which is based on \cite{ResN}, can be found in \cite{CC5} (see also \cite{Th, DH}). We point out that in \cite{CC5} it is actually proved that the image of $\mathcal F(\beta)$ through $\wt$ forms a fan covering of $\wt(\D(\beta))$. 
\end{re}

\begin{prop} \label{pure-fan} Assume that $\beta$ is a Schur root. Then $\mathcal F(\beta)$ is a pure fan of dimension $|Q_0|-1$.
\end{prop}

\begin{proof} Since $\beta$ is a Schur root, we know that $\dim_K \D(\beta)=|Q_0|-1$. Now, let $\C(\beta)_{\alpha_0} \in \mathcal F(\beta)$ be a non-zero GIT-cone. From the semi-continuity property of semi-stable loci (see for example \cite{Laza}), we know that there exists an open (with respect to the Euclidean topology) neighborhood $\mathcal U_0$ of $\alpha_0$ in $\QQ^{Q_0}$ such that:
\begin{equation} \label{cont}
\rep(Q,\beta)^s_{\wt(\alpha_0)} \subseteq \rep(Q, \beta)^s_{\wt(\alpha)} \subseteq \rep(Q, \beta)^{ss}_{\wt(\alpha)} \subseteq \rep(Q, \beta)^{ss}_{\wt(\alpha_0)},
\end{equation}
for all $\alpha \in \D(\beta) \cap \mathcal U_0$. In particular, we get that $\C(\beta)_{\alpha_0} \subseteq \C(\beta)_{\alpha}$ and hence $\C(\beta)_{\alpha_0}$ is a face of $\C(\beta)_{\alpha}$ for all $\alpha \in \D(\beta) \cap \mathcal U_0$ by Theorem \ref{git-fan}. 

Now, let us assume that $\C(\beta)_{\alpha_0}$ is maximal with respect to either dimension or inclusion. Then $\C(\beta)_{\alpha}=\C(\beta)_{\alpha_0}$ for all $\alpha \in \D(\beta) \cap \mathcal U_0$, i.e. $\alpha_0$ and $\alpha$ belong to the same GIT-class relative to $\beta$ for all $\alpha \in \D(\beta) \cap \mathcal U_0$.

Next, we claim that $\alpha_0 \in \D^0(\beta)$. Assume this is not the case. Then $\alpha_0$ lies on the boundary of $\D^0(\beta)$ and so we can choose an $\alpha \in \D^0(\beta) \cap \mathcal U_0$; in particular, $\alpha$ and $\alpha_0$ belong to the same GIT-class relative to $\beta$. Now, if $\mathcal F$ is the proper face of $\D(\beta)$ with $\alpha_0 \in \rel \mathcal F$ then $\alpha \in \mathcal F$ by Lemma \ref{L5}. But this is a contradiction since $\alpha$ is in the relative interior of $\D(\beta)$. This proves our claim above.

Finally note $\alpha_0\in \D^0(\beta) \cap \mathcal U_0$ is a non-empty open subset of the hyperplane $\mathbb H(\beta):=\{\theta \in \QQ^{Q_0}\mid \theta(\beta)=0\}$, and so it has dimension $|Q_0|-1$. Since $\C(\beta)_{\alpha_0}$ contains $\D^0(\beta) \cap \mathcal U_0$, we get that $\dim_K \C(\beta)_{\alpha_0} = |Q_0|-1$. Hence, $\mathcal F(\beta)$ is a pure fan of dimension $|Q_0|-1$.
\end{proof}

\begin{remark} The proof above shows that if $\beta$ is a Schur root then a GIT-cone $\C(\beta)_{\alpha_0}$ is maximal with respect to inclusion if and only if it is maximal with respect to dimension; if this is the case, the dimension of $\C(\beta)_{\alpha_0}$ is $|Q_0|-1$. 
\end{remark}

Now, we are ready to prove Theorem \ref{thm1}:

\begin{proof}[Proof of Theorem \ref{thm1}] 
\noindent
$(\Longrightarrow)$ Assume that $\alpha_1 \sim_{\git} \alpha_2$. Let $\mathcal C(\beta)_{\alpha} \in \mathcal I_{\alpha_1}$ where $\beta$ is a Schur root and $\alpha \in \D(\beta)$. Then $\rep(Q,\beta)^{ss}_{\wt(\alpha)} \subseteq \rep(Q, \beta)^{ss}_{\wt(\alpha_1)}$ and, since $\rep(Q, \beta)^{ss}_{\wt(\alpha_1)}=\rep(Q, \beta)^{ss}_{\wt(\alpha_2)}$ as $\alpha_1 \sim_{\git} \alpha_2$, we get that $\alpha_2 \in \mathcal C(\beta)_{\alpha}$; hence, $\C(\beta)_{\alpha} \in \mathcal I_{\alpha_2}$. This proves that $\mathcal I_{\alpha_1} \subseteq \mathcal I_{\alpha_2}$. The other inclusion is proved similarly and so $\mathcal I_{\alpha_1}=\mathcal I_{\alpha_2}$.

\noindent
$(\Longleftarrow)$ Let us assume now that $\mathcal I_{\alpha_1}=\mathcal I_{\alpha_2}$. We want to show that $\rep(Q)^{ss}_{\wt(\alpha_1)}=\rep(Q)^{ss}_{\wt(\alpha_2)}$. For this, we first check that:
\begin{equation} \label{thm1-eqn1}
\rep(Q,\beta)^{ss}_{\wt(\alpha_1)} \subseteq \rep(Q,\beta)^{ss}_{\wt(\alpha_2)},
\end{equation}
for $\wt(\alpha_1)$-stable dimension vectors $\beta$. So, let $\beta$ be a $\wt(\alpha_1)$-stable dimension vector. Then $\C(\beta) _{\alpha_1}\cap \D^0(\beta) \neq \emptyset$ since $\alpha_1$ belongs to this intersection. Furthermore, since $\beta$ is a Schur root, we know that the fan $\mathcal F(\beta)$ is pure of dimension $|Q_0|-1$ by Lemma \ref{pure-fan}. It now follows from \cite[Lemma 5]{Keicher} that: 
$$
\C(\beta)_{\alpha_1}=\bigcap_{\C} \C,
$$
where the union is over all maximal GIT-cones $\C$ of $\mathcal F(\beta)$ with $\alpha_1 \in \C$. Since $\mathcal I_{\alpha_1}=\mathcal{I}_{\alpha_2}$, any such maximal GIT-cone $\C$ contains $\alpha_2$ as well; in particular, $\C$ contains $\C(\beta)_{\alpha_2}$. So, we must have that
$\C(\beta)_{\alpha_2} \subseteq \C(\beta)_{\alpha_1}$, i.e. $\rep(Q,\beta)^{ss}_{\wt(\alpha_1)} \subseteq \rep(Q,\beta)^{ss}_{\wt(\alpha_2)}$. Similarly one checks that the opposite inclusion holds if $\beta$ is $\wt(\alpha_2)$-stable.

Next, let us check that $\rep(Q)^{ss}_{\wt(\alpha_1)} \subseteq \rep(Q)^{ss}_{\wt(\alpha_2)}$. Let $M \in \rep(Q)^{ss}_{\wt(\alpha_1)}$ and consider a Jordan-H{\"o}lder filtration of $M$ in $\rep(Q)^{ss}_{\wt(\alpha_1)}$:
$$
0=M_0<M_1< \ldots <M_n=M,
$$ 
where $M_i/M_{i-1}$ is $\wt(\alpha_1)$-stable for every $1 \leq i \leq n$. Set $\beta_i:=\dv M_i/M_{i-1}$ for all $1 \leq i \leq n$. Then we have that $\rep(Q,\beta_i)^{ss}_{\wt(\alpha_1)} \subseteq \rep(Q,\beta_i)^{ss}_{\wt(\alpha_2)}$ by $(\ref{thm1-eqn1})$. This implies that $M$ is $\wt(\alpha_2)$-semi-stable since $\rep(Q)^{ss}_{\wt(\alpha_2)}$ is closed under extensions. We have just showed that $\rep(Q)^{ss}_{\wt(\alpha_1)} \subseteq \rep(Q)^{ss}_{\wt(\alpha_2)}$. Similarly one proves that $\rep(Q)^{ss}_{\wt(\alpha_2)} \subseteq \rep(Q)^{ss}_{\wt(\alpha_1)}$, and this completes the proof.
\end{proof}

\section{The tame case} \label{tame}
Throughout this section, we assume the quiver $Q$ is tame. We denote by $\delta$ the unique isotropic Schur root of $Q$. Let $\tau$ be the Auslander-Reiten translate of $Q$. 

\subsection{Regular representations of tame quivers} For $V$ indecomposable, we say $V$ is 
\begin{itemize}
\item {\bf preprojective} if $\tau^m(V) = 0$ for some $m \gg 0$ 
\item {\bf preinjective} if $\tau^{-m}(V) = 0$ for some $m \gg 0$ 
\item {\bf regular} otherwise, that is, if $\tau^m(V) \neq 0$ for all integers $m$. 

\end{itemize}
 
For arbitrary $V$, we say that $V$ is preprojective, preinjective, or regular (respectively) if all of its indecomposable direct summands are. 

\begin{prop} \cite{CB} If $V$ is indecomposable, then $V$ is preprojective, regular, or preinjective if $\< \delta, \dv V\>$ is negative, 0, or positive (respectively). \end{prop}

\begin{prop} \cite{CB} \label{prepro} Let $V$ and $W$ be indecomposable. 

\begin{enumerate}
\item If $W$ is preprojective and $V$ is not, then $\Hom_Q(V,W) = 0$ and $\Ext_Q^1(W,V) = 0$.
\item If $W$ is preinjective and $V$ is not, then $\Hom_Q(W,V) = 0$ and $\Ext_Q^1(V,W) = 0$. 
\end{enumerate}
\end{prop}

\begin{prop} \cite{CB} The set of all regular representations of $Q$, denoted by $\mathcal{R}eg$, is an extension closed, abelian subcategory of $\rep(Q)$, and is equal to $\rep(Q)_{\wt(\delta)}^{ss}$. \end{prop}

A representation $V \in \mathcal{R}eg$ which is simple in this subcategory (that is, a $\wt(\delta)$-stable representation) will be called {\bf regular simple}. We will also call $\dv V$ regular simple where no ambiguity will result. 

\begin{prop}  \label{regsim} \cite{CB} Suppose $V$ is regular simple. 
\begin{enumerate}
\item $\tau V$ is also regular simple. 
\item $V$ is Schur, thus $\dv V$ is a Schur root. 
\item $\tau V \cong V$ if and only if $\dv V = \delta$. 
\item There exists $p \in \nat$ such that $\tau^p V \cong V$. $p$ is called the period of $V$. 
\item If $V$ has period $p$, then $\dv V + \dv \tau V + \cdots + \dv \tau^{p-1}V = \delta$.
\end{enumerate}
\end{prop}

\begin{prop} \label{regsim2} \cite{CB} Every indecomposable regular representation $V$ has a unique filtration 
$$0 =V_0 < V_1 < \cdots < V_{r-1} < V_r=V$$
such that $V_i/V_{i-1}$ is regular simple. This is called a {\bf uniserial filtration}. Furthermore, given a regular simple $V_1$ and a length $r$, there is a unique indecomposable regular representation $V$ with regular socle $V_1$ and regular length $r$. It's composition factors, from the socle, are $V_1$, $\tau^- V_1, \tau^{-2}V_1, \ldots, \tau^{1-r}V_1$. 
\end{prop}

Now, we can organize all indecomposable regular representations into groups, called {\bf tubes}, in the following way: Begin with a regular simple representation $V$ and set $\mathcal{V} = \{V, \tau V, \ldots, \tau^{p-1} V\}$, the $\tau$-orbit of $V$. Then, include in the tube any indecomposable regular representation whose composition factors are taken from $\mathcal{V}$. We say that $p$ is the period of the tube.

\begin{prop} \cite{CB} Every regular indecomposable representation belongs to a unique tube. \end{prop} 

As a consequence of Proposition \ref{regsim}, we see that the dimension vectors of the regular simples must be either real Schur roots, or must be equal to $\delta$. There are infinitely many (occurring in families) $\delta$-dimensional indecomposable representations. Each one generates a tube of period 1, that is, a homogeneous tube. Real Schur roots correspond to exceptional representations, so we have finitely many non-homogeneous tubes generated by exceptional representations which are also regular simple. For the purpose of distinguishing from $\delta$, call the dimension vectors of exceptional regular simple representations {\bf quasi-simple}. 

If $\beta$ is the dimension vector of a regular representation of $Q$, we define $\tau \beta$ as $\beta \Phi$ and $\tau^- \beta$ as $\beta \Phi^{-1}$, where $\Phi$ is the Coxeter matrix of $Q$. If  $V$ is regular, then $\dv(\tau V)=\tau \dv V$, and $\dv(\tau^- V)=\tau^{-}\dv V$  (see \cite{BB1}). 

\begin{lemma} If $\beta_1$ and $\beta_2$ are quasi-simple, then 

$$\<\beta_1,\beta_2\> = \left\{ \begin{array}{l} 1 \mbox{ if } \beta_1=\beta_2 \\ -1 \mbox{ if } \beta_2 = \tau \beta_1 \\ 0 \mbox{ otherwise} \end{array} \right. $$

\end{lemma}

\begin{proof} Let $V_i$ be the unique indecomposable representation of dimension $\beta_i$ for $i = 1,2$. If $\beta_1 = \beta_2$, then $\<\beta_1,\beta_2\>=q(\beta_1)=1$ since $\beta_1$ is a real Schur root. 

If $\beta_2 = \tau \beta_1$, then since $V_1 \ncong V_2$, and $V_1$ and $V_2$ are simple in $\rep(Q)_{\<\delta, \cdot\>}^{ss}$, we have that $\Hom_Q(V_1,V_2)=0$. So 
$\<\beta_1, \beta_2\> = - \dim_K \Ext_Q^1(V_1,\tau V_1) = - \dim_K \Hom_Q(\tau V_1, \tau V_1) = -1$.

If $\beta_2 \neq \beta_1, \tau\beta_1$, we get $\<\beta_1, \beta_2\> =  - \dim_K \Ext_Q^1(V_1,V_2) =  -\dim_K \Hom_Q(V_2, \tau V_1) = 0$. 

\end{proof}

\subsection{A Summary of Work by Ingalls, Paquette, and Thomas} We start by defining a rational convex polyhedral cone $H_{\delta}^{ss}$ to be the cone of non-negative rational linear combinations of regular simple dimension vectors. The generators of this cone are $\delta$, and the quasi-simple roots, which we recall are divided into finitely many (say $N$) tubes. For convenience, label the quasi-simple roots $\beta_{i,j}$, so that $\beta_{i,j+1} = \tau\beta_{i,j}$, and $i = 1,\ldots, N$ indicates which tube the root is in. Call the rank, or period, of the $i^{th}$ tube $r_i$. 

\begin{re} It turns out that $H_{\delta}^{ss} = \D(\delta)$. This follows from Lemma \ref{simples} and the fact that $\<\delta, \cdot\> = -\<\cdot, \delta\>$. \end{re}

Next, define the following cover of $H_{\delta}^{ss}$. For any multi-index $I = (a_1, \ldots, a_N)$, with $1\leq a_i \leq r_i$, $C_I$ is the cone given by non-negative rational linear combinations of $\delta$, and all quasi-simple roots {\it except} $\beta_{1,a_1}, \ldots, \beta_{N,a_N}$. Let $R$ be the set of all such multi-indices and let
$$\mathcal{J} = \{C_I\}_{I \in R} \cup \{\D(\beta)\}_{\beta},$$
where $\beta$ runs through the set of real Schur roots of $Q$. Furthermore, for any $\alpha \in \inte^{Q_0}$ define $\mathcal{J}_{\alpha} = \{J \in \mathcal{J} | \alpha \in J\}$. 

\begin{re}  \begin{enumerate} \item The above definition is structured for a Euclidean quiver. If $Q$ is Dynkin,  $\mathcal{J} = \{\D(\beta)\}_{\beta}$ where $\beta$ runs through all (real Schur) roots. 
\item Also see \cite[Proposition 5.2]{HidelaP}, in which the authors investigate a wall system for $D(\delta)$.
\end{enumerate} \end{re} 

The theorem we will recover in the tame case is as follows:

\begin{thm*} \cite[Theorem 7.4]{IngPacTho} \label{IPT} Let $Q$ be a tame quiver and $\alpha_1, \alpha_2 \in \inte^{Q_0}$. Then $\alpha_1\sim \alpha_2$ if and only if $\mathcal{J}_{\alpha_1} = \mathcal{J}_{\alpha_2}$. \end{thm*}

\subsection{Proof of Theorem 2}

Let $I$ be a multi-index as above, and let $\wt(\alpha_I)$ be the weight given by $\alpha_I = \delta +  \sum_{j \neq a_i} \sum_{i=1}^N  \beta_{i,j}$. We want to prove the following two substantial Lemmas, which together will, in large part, prove Theorem 2:

\begin{lemma} \label{big1} $C_I = \C(\delta)_{\alpha_I}$, and $\C(\delta)_{\alpha_I}$ is a maximal GIT-cone. \end{lemma}

\begin{lemma} \label{big2} If $Q$ is Euclidean, any maximal GIT-cone $\C(\beta)_{\alpha}$ (with $\beta$ Schur) is either $\D(\beta)$ or $C_I$ for some multi-index $I$, if $\beta$ is real or isotropic, respectively. \end{lemma}

The proof of these two lemmas will be delayed, as we need some auxiliary results first. We start with a technical lemma:

\begin{lemma} \label{tech} Given an indecomposable regular representation $V$, and a quasi-simple root $\beta_{i,j}$, to check that $V$ is $\wt(\beta_{i,j})$-semi-stable, it is enough to check that $\<\beta_{i,j}, \dv V\> = 0$, and $\<\beta_{i,j}, \dv V'\> \leq 0$ for all regular subrepresentations $V'$ of $V$. Furthermore, for the weight $\alpha_I = \delta + \sum_{j \neq a_i} \sum _{i=1}^N \beta_{i,j}$, to check that $V$ is $\wt(\alpha_I)$-stable, it is enough to check that $\<\alpha_I, \dv V\> = 0$ and $\<\alpha_I, \dv V' \> < 0$ for all proper, non-zero regular subrepresentations $V'$ of $V$. 
\end{lemma}

\begin{proof} Suppose we have indeed checked the criteria indicated, and let $Y \leq V$ be a non-regular subrepresentation. Then the indecomposable direct summands of $Y$ are either preprojective or regular. Without loss of generality, we can assume $Y$ is preprojective and indecomposable. Let $B_{i,j}$ be the unique exceptional representation of dimension vector $\beta_{i,j}$. Then:
\begin{equation} \label{proj} \<\beta_{i,j}, \dv Y\> = \dim_K\Hom_Q(B_{i,j},Y) - \dim_K \Ext_Q^1(B_{i,j},Y).\end{equation}Now since $B_{i,j}$ is not preprojective (it is regular), and $Y$ is preprojective, by Proposition \ref{prepro}, we have $\Hom_Q(B_{i,j},Y) = 0$. So, $\<\beta_{i,j}, \dv Y\> = - \dim_K \Ext_Q^1(B_{i,j},Y) \leq 0$. Hence, $V$ is $\wt(\beta_{i,j})$-semi-stable. Furthermore, 

$$\<\alpha_I, \dv Y\> = \<\delta, \dv Y\> + \sum_{j \neq a_i} \sum_{i=1}^N \<\beta_{i,j}, \dv Y\> < 0,$$ since $\< \delta, \dv Y\> <0$, for $Y$ preprojective, and $\<\beta_{i,j}, \dv Y\> \leq 0$ by equation (\ref{proj}). So $V$ is $\wt(\alpha_I)$-stable. \end{proof}

\noindent Recall that the orbit cone of any representation $V$ is defined as $$\Omega(V) = \{\alpha \in \rat^{Q_0} | V \in \rep(Q)_{\wt(\alpha)}^{ss}\}.$$ We have the following straightforward but useful lemma:

\begin{lemma} \label{directsum} If $V = \bigoplus_{i=1}^m V_i$ then $\Omega(V) = \bigcap_{i=1}^m \Omega(V_i)$. 
\end{lemma}

\begin{proof} Suppose $\alpha \in \Omega(V)$. Then we have $\<\alpha, \dv V_i\> \leq 0$ since $V_i \leq V$ and $V$ is $\wt(\alpha)$-semi-stable. On the other hand,
$$\<\alpha, \dv V\> = \sum_{i=1}^m\<\alpha, \dv V_i\> = 0,$$
so $\<\alpha, \dv V_i\> = 0$ for all $i = 1, \ldots, m$. Next, for each $i$, a subrepresentation $V'_i$ of $V_i$ is a subrepresentation of $V$ and, therefore, $\<\alpha, \dv V'_i\> \leq 0$. So, $\alpha \in \Omega(V_i)$ for each $1 \leq i \leq n$. 

Conversely, if $\alpha \in \Omega(V_i)$ for all $i = 1,\ldots, m$, we immediately have $\alpha \in \Omega(V)$ since $\rep(Q)_{\wt(\alpha)}^{ss}$ is closed under extensions.  \end{proof}

Recall now the multi-index $I = (a_1,\ldots, a_N)$, and the corresponding list \{$\beta_{i,a_i}\}_{i=1}^N$ of quasi-simple roots, one from each tube of the Auslander-Reiten quiver of $Q$. By Proposition \ref{regsim}, we know that there exists a unique $\delta$-dimensional indecomposable representation $Z_i$ of length $r_i$ and whose regular socle is of dimension $\beta_{i,a_i}$. Following the notation from Ingalls-Paquette-Thomas in \cite{IngPacTho}, define $Z_I = \bigoplus_{i=1}^N Z_i$.

\begin{lemma} (compare with \cite[Lemma 7.5]{IngPacTho}) \label{lemma-Z_I} Keeping the notation above, $C_I = \Omega(Z_I)$. 

\end{lemma}

\begin{proof} We will show that all generators of $C_I$, that is, $\delta$, and $\beta_{i,j}$ except $\beta_{i,a_i}$, are in $\Omega(Z_i)$ for all $i$ (see (a) and (b) below). This implies that $C_I \subseteq \Omega(Z_i)$ for all $i$. Then, by Lemma \ref{directsum}, $\Omega(Z_I) = \bigcap_{i=1}^N \Omega(Z_i)$, so $C_I \subseteq \Omega(Z_I)$. Next, of course, we will show the reverse inclusion (see (c) below). \\

\noindent (a) $\delta$ is certainly in $\Omega(Z_i)$, since each $Z_i$ is regular. \\

\noindent (b) Let $\beta\neq \delta$ be a generator for $C_I$. Since $\beta$ is quasi-simple, we have:
$$\<\beta, \dv Z_i\> = \<\beta,\delta\> = -\<\delta,\beta\> = 0,$$
for every $1 \leq i \leq N$. Now, by Lemma \ref{tech}, to show that $\beta \in \Omega(Z_i)$, it is enough to check that $\<\beta, \dv Z_i'\> \leq 0$ for any regular subrepresentation $Z_i'\leq Z_i$. Let 
$$0 = Z_{i,0} < Z_{i,1} < \cdots < Z_{i,r_i}=Z_i$$ be the uniserial  filtration for $Z_i$. We have that $\dv Z_{i,1} = \beta_{i,a_i}$ by assumption, and $Z_{i,j}$, $j = 1, \ldots, r_i-1$, is a complete list of the (proper) regular subrepresentations of $Z_i$. Recall that $\dv (Z_{i,r}/Z_{i,r-1}) = \tau^{-(r-1)}(\beta_{i,a_i})$.

We will check by induction on $r$ that $\<\beta, \dv Z_{i,r}\> \leq 0$ for all $1 \leq r \leq r_i$. Now, since  $\beta \neq \beta_{i,a_i}$, we get that:
$$\<\beta, \dv Z_{i,1}\> = \left\{ \begin{array}{lr} -1 & \mbox{ if } \beta = \tau^-(\beta_{i,a_i}) \\ 0 & \mbox{ otherwise - including $\beta$ in different tube, }\end{array} \right.$$ Assume now that $\<\beta, \dv Z_{i,s}\> \leq 0$ for all $s=1, \ldots, r-1$. We have that:
$$\<\beta,\dv Z_{i,r}\> = \<\beta, \dv Z_{i,r}/Z_{i,r-1}\> + \<\beta, \dv Z_{i,r-1}\> $$

$$=\<\beta, \dv Z_{i,r}/Z_{i,r-1}\> + \<\beta,\dv Z_{i,r-1}/Z_{i,r-2}\> + \<\beta,\dv Z_{i,r-2}\>.$$
Furthermore, $$\<\beta, \dv Z_{i,r}/Z_{i,r-1}\> = \left\{ \begin{array}{lr} 1 & \mbox{ if } \beta = \dv Z_{i,r}/Z_{i,r-1} \\ -1 & \mbox{ if } \beta = \tau^-(\dv Z_{i,r}/Z_{i,r-1}) \\ 0 & \mbox{ otherwise }\end{array} \right.$$

\noindent {\bf Case 1}: If $\beta = \dv Z_{i,r}/Z_{i,r-1}$, then $\beta = \tau^-(\dv Z_{i,r-1}/Z_{i,r-2})$, and  $\<\beta,\dv Z_{i,r-2}\> \leq 0$ by inductive assumption. So $\<\beta,\dv Z_{i,r}\> = 1 + (-1) + \<\beta,\dv Z_{i,r-2}\> \leq 0$.
\\

\noindent {\bf Case 2}: If $\beta = \tau^-(\dv Z_{i,r}/Z_{i,r-1})$, then $$\<\beta,\dv Z_{i,r}\> = \<\beta, \dv Z_{i,r}/Z_{i,r-1}\> + \<\beta, \dv Z_{i,r-1}\> = -1 +  \<\beta, \dv Z_{i,r-1}\> \leq 0$$
since $ \<\beta, \dv Z_{i,r-1}\> \leq 0$ by inductive assumption.
\\

\noindent {\bf Case 3}: Otherwise, if $\<\beta, \dv Z_{i,r}/Z_{i,r-1}\> = 0$, then $$\<\beta,\dv Z_{i,r}\> = \<\beta, \dv Z_{i,r}/Z_{i,r-1}\> + \<\beta, \dv Z_{i,r-1}\> = 0 +  \<\beta, \dv Z_{i,r-1}\> \leq 0$$
again by inductive assumption. \\

We have just proved $\<\beta, \dv Z_{i,r}\> \leq 0$ for $r = 1, \ldots, r_i$, so $Z_i$ is $\wt(\beta)$-semi-stable. That is, $\beta \in \Omega(Z_i)$ for all $i$. \\

\noindent (c) We want to show $\bigcap_{i=1}^{N} \Omega(Z_i) \subseteq C_I$. We do not know, a priori, generators for $\Omega(Z_i)$, but we know $\bigcap_{i=1}^{N} \Omega(Z_i) \subset \D(\delta)$. So, any $\alpha \in \bigcap_{i=1}^{N} \Omega(Z_i)$ can be written as

$$\alpha = c\delta + \sum_{i=1}^N\sum_{j=1}^{r_i} c_{i,j}\beta_{i,j}.$$
since the generators of $D(\delta)$ are $\delta$ plus the quasi-simple roots. We want to show that the coefficient $c_{i,a_i} = 0$ for all $1 \leq i \leq N$.

Now, there are linear dependencies among the $\beta_{i,j}$'s and $\delta$, namely $\sum_{j=1}^{r_i} \beta_{i,j} = \delta$. Thus, we may take the above expression of $\alpha$ to be minimal in the following sense: For each $i$, at least one coefficient $c_{i,j}$ must be 0. 

Now, let $$0 = Z_{i,0} < Z_{i,1} < \cdots < Z_{i,r_i}=Z_i$$ be the uniserial  filtration for $Z_i$. We have: 
$$\dv Z_{i,1} = \beta_{i,a_i} \mbox{ (by construction of } Z_i) = \beta_{i,j_1} \mbox{ (for notational convenience)}$$
$$\dv Z_{i,2}/Z_{i,1} = \tau^-\beta_{i,a_i} = \beta_{i,j_2}$$
$$\vdots$$
$$\dv Z_{i,l}/Z_{i,l-1} = \tau^{-(l-1)}\beta_{i,a_i} = \beta_{i,j_{l}}$$
$$\vdots$$
$$\dv Z_i/Z_{i,r_i-1} = \tau^{-(r_i-1)}\beta_{i,a_i} = \beta_{i,j_{r_i}}.$$
In other words, the composition factors of $Z_i$ have dimensions (starting from the socle and going up): $\beta_{i,j_1},\beta_{i,j_2}, \ldots, \beta_{i,j_{r_i}}$. 

We know that $\<\alpha, \dv Z_{i,l}\> \leq 0$, $\forall l = 1, \ldots, r_i, \forall i=1, \ldots, N$, since $\alpha \in \bigcap_{i=1}^N \Omega(Z_i)$. So we get:

$$\<\alpha, \dv Z_{i,1}\> = \<\alpha, \dv \beta_{i,j_1}\> = c_{i,j_1} - c_{i,j_2} \leq 0 \Rightarrow c_{i,j_1} \leq c_{i,j_2},$$

$$\<\alpha, \dv Z_{i,2}\> = \<\alpha, \beta_{i,j_2}\> + \<\alpha, \dv Z_{i,1}\> = c_{i,j_2} - c_{i,j_3} + c_{i,j_1}-c_{i,j_2} = c_{i,j_1}-c_{i,j_3} \Rightarrow c_{i,j_1} \leq c_{i,j_3}.$$

Similarly, 

$$\<\alpha,\dv Z_{i,l}\> = c_{i,j_1} - c_{i,j_{l+1}} \Rightarrow c_{i,j_1} \leq c_{i,j_{l+1}} \mbox{ for } l = 1, \ldots, r_i-1.$$

Now, at least one of the $c_{i,j}$'s is 0, and we just showed $c_{i,j_1}$ is minimal, so $c_{i,j_1} = c_{i,a_i} = 0$ for each $i$. Thus, $\alpha$ is a sum of generators of $C_I$, so $\alpha \in C_I$. \end{proof}

\begin{corollary} \label{stable} The $Z_i$ as defined above is $\wt(\alpha_I)$-stable for all $1 \leq i \leq N$. \end{corollary}

\begin{proof} We already know they are $\wt(\alpha_I)$-semi-stable. Let 
$$0 = Z_{i,0} < Z_{i,1} < \cdots < Z_{i,r_i}=Z_i$$ be the uniserial  filtration for $Z_i$. By Lemma \ref{tech}, it is enough to check that $\<\alpha_I, \dv Z_{i,r}\><0$ for all $1 \leq r <r_i$.

Recall that $\dv (Z_{i,r}/Z_{i,r-1}) = \tau^{-(r-1)}(\beta_{i,a_i})$. Also note that the weight $\alpha_I$ is 
\begin{equation} \label{tau} \alpha_I = \delta + \tau^{-}\beta_{i,a_i} + \cdot + \tau^{-(r_i-1)}(\beta_{i,a_i}) + \mbox{quasi-simples from other tubes}
\end{equation} 
We will proceed by induction on $r$ to show that $\<\alpha_I, Z_{i,r}\> = -1<0$ for all $1 \leq r <r_i$.  

When $r=1$: $$\<\alpha_I, \dv Z_{i,1}\> = \<\alpha_I, \beta_{i,a_i}\> = -1,$$ from (\ref{tau}), since $\tau^-\beta_{i,a_i}$ appears in the sum for $\alpha_I$, but $\beta_{i,a_i}$ does not.

\noindent Assume now that $\<\alpha_I, \dv Z_{i,s}\> = -1$ for $s = 1, \ldots, r-1$. Then 

$$\<\alpha_I, \dv Z_{i,r}/Z_{i,r-1}\> = \<\alpha_I, \dv Z_{i,r}\>-\<\alpha_I, \dv Z_{i,r-1}\>$$ and so $$\<\alpha_I, \tau^{-(r-1)}(\beta_{i,a_i})\> = \<\alpha_I, \dv Z_{i,r}\>+1.$$ Then using (\ref{tau}) again, since $r<r_i$, we have

$$0 = \<\alpha_I, \dv Z_{i,r}\>+1.$$ That is, $ \<\alpha_I, \dv Z_{i,r}\> = -1$. So, $Z_i$ is $\wt(\alpha_I)$-stable. \end{proof}

\begin{lemma} \label{reg} If $V$ is  an indecomposable representation of $Q$ with $\dv V < \delta$, $\<\alpha_I, \dv V\> = 0$, then $V$ is regular. \end{lemma}

\begin{proof} The weight in question is 

$$\alpha_I = \delta + \sum \beta_{i,j}.$$

For this proof it is not particularly important which $\beta_{i,j}$'s are included in the sum. For our specific weight $\alpha_I$, the sum is over all $i$, and all $j \neq a_i$. 

Suppose $V$ was preinjective. Then we have $\<\delta, \dv V\>  >0$, so $ \left\< \sum \beta_{i,j}, \dv V \right\> <0$, since $\<\alpha_I, \dv V\> = 0$. Now $\beta_{i,j}$ is a quasi-simple root (in particular it is a real Schur root), so there is a unique indecomposable $B_{i,j}$ with $\dv B_{i,j} = \beta_{i,j}$. We have 

$$\sum \left( \dim_K \Hom_Q(B_{i,j}, V) - \dim_K \Ext_Q^1(B_{i,j},V)\right) < 0.$$
But $B_{i,j}$ is regular, while $V$ is preinjective, so $\dim_K \Ext_Q^1(B_{i,j}, V) = 0$ for all $i,j$ by Proposition \ref{prepro}. Thus we conclude $\sum \dim_K \Hom_Q(B_{i,j}, V) < 0$, which is a contradiction. The case that $V$ is preprojective is similar. So, $V$ must be regular. \end{proof}

\begin{lemma} \label{sametube} If $V = \bigoplus_{l=1}^mV_l$ is the decomposition of $V$ into indecomposables, $V$ is $\delta$-dimensional, and $\<\alpha_I, \dv V_l\> = 0$ for each $l$, then the $V_l$'s all appear in the same non-homogeneous tube of the Auslander-Reiten quiver of $Q$. \end{lemma}

\begin{proof} We know each $V_l$ is regular by Lemma \ref{reg}. The dimension of each $V_l$ can be written as a sum of ($\tau$-consecutive) quasi-simples $\beta_{i,j}$, using the composition factors of the uniserial filtration for $V_l$. That is, we have $\sum_{j \in J_i} \beta_{i,j} = \dv V_l$ if $V_l$ is in the $i$th tube. Since $\sum_{l=1}^m \dv V_l = \delta$, we have: 

\begin{equation} \label{sum} \sum_{j \in J_1} c_{1,j}\beta_{1,j} + \cdots + \sum_{j \in J_N} c_{N,j}\beta_{N,j} = \delta,
\end{equation}
where $J_i$ is some subset of $\{1, \ldots, r_i\}$.

Note that we are not assuming there is one $V_l$ in each tube, we are merely combining quasi-simples from the same tube into a single sum. There may be constants, thus the $c_{i,j}$'s, if $V$ has repeated direct summands, or if different $V_l$'s share composition factors. None of the constants are 0, by construction, since we are not including any trivial summands. Note that at least one of the constants must be 1, since $\delta(0) = 1$ where $0$ is a vertex of $Q$ such that $Q\setminus \{0\}$ is a Dynkin diagram.  Exactly one quasi-simple from each tube has $\beta_{i,j}(0) = 1$, and exactly one of these must appear in the sum $(\ref{sum}$). Without loss, assume this occurs for the first tube. That is, there is a $j_0$ such that $\beta_{1,j_0}(0) = 1$, and $c_{1,j_0} =1$ in the sum $(\ref{sum})$. 

Set $J_1' = J_1\setminus \{j_0\}$, so we have 

$$\delta = \beta_{1,j_0} + \sum_{j \in J_1'} c_{1j}\beta_{1,j} = \sum_{j \in J_2} c_{2j}\beta_{2,j} + \cdots + \sum_{j \in J_N} c_{N,j}\beta_{N,j},$$

$$\delta - \beta_{1,j_0} = \sum_{j \in J_1'} c_{1,j}\beta_{1,j} = \sum_{j \in J_2} c_{2,j}\beta_{2,j} + \cdots + \sum_{j \in J_N} c_{N,j}\beta_{N,j}.$$

Using the relations among quasi-simples, 

\begin{equation} \label{sum2} \sum_{j \neq j_0} \beta_{1,j} = \sum_{j \in J_1'} c_{1,j}\beta_{1,j} + \sum_{j \in J_2} c_{2,j}\beta_{2,j} + \cdots + \sum_{j \in J_N} c_{N,j}\beta_{N,j}.
\end{equation}

Now, taking the Euler inner product of each side with $\beta_{1,j_0}$, we see
$$\<\beta_{1,j_0}, \sum_{j \neq j_0} \beta_{1,j} \> = -1.$$

On the other hand,

$$\<\beta_{1,j_0}, \sum_{j \in J_1'} c_{1,j}\beta_{1,j} + \sum_{j \in J_2} c_{2,j}\beta_{2,j} + \cdots + \sum_{j \in J_N} c_{N,j}\beta_{N,j}\> = \<\beta_{1,j_0}, \sum_{j \in J_1'} c_{1,j}\beta_{1,j} \>.$$
If $\tau\beta_{1,j_0} = \beta_{1,j_1}$ does not appear in the right hand sum, that is, if $j_1 \notin J_1'$, this inner product is 0, so we have a contradiction, and the proof is complete. If it does appear, the inner product is $-c_{1,j_1}$, so we have $c_{1,j_1} = 1$.  In that case, we return to equation $(\ref{sum2})$, add the new information, and we have

\begin{equation} \label{sum3} \sum_{j \neq j_0, j_1} \beta_{1,j} = \sum_{j \in J_1''} c_{1,j}\beta_{1,j} + \sum_{j \in J_2} c_{2,j}\beta_{2,j} + \cdots + \sum_{j \in J_N} c_{N,j}\beta_{N,j}
\end{equation}

\noindent where $J_1'' = J_1' \setminus\{j_1\}$. If $\tau \beta_{1,j_1} = \beta_{1,j_2}$ appears in the sum (that is, if $j_2 \in J_1''$), then $c_{1,j_2} = 1$ as well. Continuing in this fashion, either we arrive at a contradiction that completes the proof, or $c_{1,j} = 1$ for all $j \in J_1$. So, we have reduced to the following: 

\begin{equation} \label{sum4} \sum_{j \in J_1} \beta_{1,j} + \sum_{j \in J_2} c_{2,j}\beta_{2,j} + \cdots + \sum_{j \in J_N} c_{N,j}\beta_{N,j} = \delta.
\end{equation}

We want to prove all but one of these sums is trivial. Since we have already labeled the tubes in a manner to make the first sum non-trivial, we want to show the others are indeed trivial. Assume the $V_l$'s are not all in the same tube. That is, assume at least two of the sums in equation (\ref{sum4}) are non-trivial. We have:

$$\delta - \sum_{j \in J_1}\beta_{1,j} = \sum_{j \in J_2} c_{2j}\beta_{2,j} + \cdots + \sum_{j \in J_N} c_{N,j}\beta_{N,j} \Longleftrightarrow $$

$$\sum_{j \in J_1^c}\beta_{1,j} = \sum_{j \in J_2} c_{2,j}\beta_{2,j} + \cdots + \sum_{j \in J_N} c_{N,j}\beta_{N,j}.$$

Note that both $J_1$ and $J_1^c$ are non-empty. We claim that we can find a $\beta_{1,j_0}$ with $j_0 \in J_1^c$ such that $\tau\beta_{1,j_0} = \beta_{1,j_0'}$ where $j_0' \in J_1$. Indeed, if it were not, that is if $\beta_{1,j_0}$ with $j_0 \in J_1^c$ implied $j_0' \in J_1^c$, where $\beta_{1,j_0'} = \tau\beta_{1,j_0}$, then by applying $\tau$ successively, we would get $J_1^c = \{1,\ldots, r_1\}$, and $J_1$ empty. 

Now, we have $\< \beta_{1,j_0}, \sum_{j \in J_1^c}\beta_{1,j}\> = -1$ since $\<\beta_{1,j_0}, \tau\beta_{1,j_0}\> = -1$, while $$\< \beta_{1,j_0},\sum_{j \in J_2} c_{2,j}\beta_{2,j} + \cdots + \sum_{j \in J_N} c_{N,j}\beta_{N,j}\> = 0$$
which is a contradiction. So, it must be that all $V_l$'s are in the same tube. \end{proof}

Given a rational weight $\theta \in \QQ^{Q_0}$, a representation $W \in \rep(Q)$ is said to be {\bf $\theta$-polystable} if and only if the indecomposable direct summands of $W$ are $\theta$-stable. A key result we will use says that if $\alpha \in \D(\beta)$ then

\begin{equation} \label{CC5} \C(\beta)_{\alpha} = \bigcap_{W} \Omega(W),
\end{equation}
where the intersection is over all $\wt(\alpha)$-polystable representations $W \in \rep(Q, \beta)$ (see \cite{CC5} for a proof).

\begin{lemma} \label{actually} If $V$ is $\delta$-dimensional and $\wt(\alpha_I)$-polystable, then $V$ is actually indecomposable and thus $\wt(\alpha_I)$-stable. 
\end{lemma}

\begin{proof} Suppose by way of contradiction that a decomposable such $V$ exists. Set $V= \bigoplus_{l=1}^mV_l$ to be the decomposition into indecomposables. Since $V$ is $\wt(\alpha_I)$-polystable, each $V_l$ is $\wt(\alpha_I)$-stable. Applying Lemma \ref{reg}, each $V_l$ is regular, and applying Lemma \ref{sametube}, they all lie in the same tube. 

Each $V_l$ has a uniserial filtration, so that $\dv V_l$ is a sum of quasi-simple roots. Since \\ $ \sum_{l=1}^m \dv V_l = \delta$, at most one of the $V_l$'s may have regular socle of dimension $\beta_{i,a_i}$. Since we are assuming $V$ is decomposable, there are at least 2 direct summands, and we can choose $V_{l_o}$ so that the regular socle of $V_{l_o}$ is of dimension $\beta_{i,j}$ with $j \neq a_i$. But then $\< \alpha_I, \beta_{i,j}\> = 0$ or 1, a contradiction to $V_l$ being $\wt(\alpha_I)$-stable. \end{proof}

\begin{lemma} \label{dan} If $V$ is $\delta$-dimensional and homogeneous (i.e., regular simple), then $\Omega(V) = \D(\delta)$. \end{lemma}

\begin{proof} Certainly $\Omega(V) \subseteq \D(\delta)$. Suppose $\alpha \in \D(\delta)$, and let $0 \neq V' <V$. Since $V$ is regular simple, it has no proper regular subrepresentations. So $V'$ must be preprojective. Without loss, assume $V'$ is indecomposable. 

If $\dv V' \into \delta$, then $\< \alpha, \dv V'\> \leq 0$ since $\alpha \in \D(\beta)$, so $\alpha \in \Omega(V)$. If not, then by \cite[Theorem 2.7]{DW2} we know $\Ext_Q^{1}(V', V/V') \neq 0$. Let $V/V' = \bigoplus_{i=1}^m V_i$ be the decomposition of $V/V'$ into indecomposable representations. Since $V$ is regular, by Proposition \ref{prepro} $V/V'$ has no preprojective direct summand. So, by our assumption we have $\bigoplus_{i=1}^m \Ext_Q^1(V', V_i) \neq 0$. But, $V'$ is preprojective and $V_i$ is not, so $\Ext_Q^1(V', V_i) =0$ by Proposition \ref{prepro}, which is a contradiction. \end{proof}

Recall that $Z_i$ is the unique $\delta$-dimensional indecomposable representation of $Q$ with regular socle of dimension $\beta_{i,a_i}$. 

\begin{lemma} \label{list} The list of $\delta$-dimensional, $\wt(\alpha_I)$-polystable representations is exactly: $Z_i$ for each $i = 1, \ldots, N$, together with the homogeneous Schur representations of $Q$. 
\end{lemma}

\begin{proof} First of all, we know from Lemma \ref{actually} that any $\delta$-dimensional, $\wt(\alpha_I)$-polystable representation is in fact regular indecomposable, thus $\wt(\alpha_I)$-stable. 

Next, the $Z_i$'s are $\wt(\alpha_I)$-stable by Corollary \ref{stable}. On the other hand, if $V$ is any other $\delta$-dimensional indecomposable from a non-homogeneous tube, thus not isomorphic to one of the $Z_i$'s, then the regular socle of $V$ is of dimension $\beta_{i,j}$ for some $j \neq a_i$ and $1 \leq i \leq N$. So $V$ has a subrepresentation $V'$ with $\<\alpha_I, \dv V'\> = \<\alpha_I, \beta_{i,j}\> = 0$ or $1$, so $V$ is not $\wt(\alpha_I)$-stable. 

Finally, let $Z$ be an indecomposable representation of dimension $\delta$ lying in a homogeneous tube. Then $Z$ has no regular subrepresentations since it is regular simple. So, there is nothing to check to verify that $\<\alpha_I, Z'\> <0$ for all regular subrepresentations. Hence, $Z$ is $\wt(\alpha_I)$-stable by Lemma \ref{tech}. 
\end{proof}

We are now ready to prove our results from the beginning of this section. 

\begin{proof}[Proof of Lemma \ref{big1}] 
We have shown that $C_I = \Omega(Z_I)$ in Lemma \ref{lemma-Z_I}. Next, using Lemma \ref{list} and equation (\ref{CC5}), we get that:
$$\Omega(Z_I) = \bigcap_{i=1}^N \Omega(Z_i) = \left(\bigcap_{i=1}^N\Omega(Z_i)\right) \cap D(\delta) = \bigcap_{\begin{array}{c}{\tiny X \mbox{ is }\delta\mbox{-dimensional }} \\{\tiny \wt(\alpha_I)\mbox{-polystable}}\end{array}} \Omega(X)=\C(\delta)_{\alpha_I}.$$ 

Lastly, $C_I$ is of maximal dimension, so $\C(\delta)_{\alpha_I}$ is of maximal dimension as well. 
\end{proof}

\begin{proof}[Proof of Lemma \ref{big2}] If $\beta$ is real, then $\C(\beta)_{\alpha}$ maximal implies $\C(\beta)_{\alpha} = \D(\beta)$. If $\beta$ is not real, then it must be equal to $\delta$, since it is Schur and $Q$ is Euclidean. Note that $\alpha \in \C(\delta)_{\alpha} \subset \D(\delta)$. The $C_I$'s cover $\D(\delta)$ as $I$ varies, thus $\alpha \in C_I$ for some $I$. We already know $C_I = \C(\delta)_{\alpha_I}$ for the weight $\alpha_I$, which is a maximal GIT-cone. So, we have $\alpha \in (\relint \C(\delta)_{\alpha}) \cap \C(\delta)_{\alpha_I}$ where both cones are maximal. Thus, $\C(\delta)_{\alpha}$ and $\C(\delta)_{\alpha_I}$ must be equal.  
\end{proof} 

We are now ready to prove Theorem \ref{thm2}, which together with Lemmas \ref{big1} and \ref{big2} says that $\mathcal{J} = \mathcal{I}$. 

\begin{proof}[Proof of Theorem 2] From Lemmas \ref{big1} and \ref{big2} we get the desired description of the set $\mathcal{I}$. 

In what follows, we show that $\mathcal{I}=\mathcal{J}$. Suppose $J \in \mathcal{J}$. Then $J = C_I$ for some multi-index $I$, or $J = \D(\beta)$ for some real Schur root $\beta$. If $J = C_I$, then $J = \Omega(Z_I) = \C(\delta)_{\alpha_I} \in \C(\delta)$, and thus $\C(\delta)_{\alpha_I} \in \mathcal{I}$. If $J = \D(\beta)$, for a real Schur root $\beta$, then $J \in \mathcal{I}$. So $\mathcal{J} \subseteq \mathcal{L}$. 

Conversely, suppose $\C \in \mathcal{I}$. Then $\C = \C(\beta)_{\alpha}$ for some Schur root $\beta$ and some weight $\wt(\alpha)$. Since $Q$ is Euclidean, $\beta$ is either real, or $\beta = \delta$. If $\beta$ is real, $\C=\C(\beta)_{\alpha}$ maximal implies that $\C = \D(\beta) \in \mathcal{J}$. If $\beta = \delta$, $\C(\delta)_{\alpha}$ maximal implies $\C = C_I$ for some $I$ by Lemma \ref{big2}. So, $\mathcal{I} \subseteq \mathcal{J}$. The proof now follows.
\end{proof}


\end{document}